\documentclass[11pt]{amsart}
\usepackage{latexsym}
\usepackage{amsmath,amsthm,amssymb, amscd,color}
\usepackage{mathrsfs}
\usepackage[all]{xy}
\usepackage[all]{xy}
\usepackage{tikz}
\usepackage{tikz-cd}
\usepackage{adjustbox}
\usetikzlibrary{snakes}

\usepackage{blkarray} 
\usepackage{framed}
\theoremstyle{plain}
\newtheorem{theorem}{Theorem}[section]
\newtheorem{lemma}[theorem]{Lemma}

\numberwithin{equation}{section}

\theoremstyle{definition}
\newtheorem{definition}[theorem]{Definition}
\newtheorem{example}[theorem]{Example}

\newtheorem{proposition}[theorem]{Proposition}
\newtheorem{remark}[theorem]{Remark}

\theoremstyle{remark}

\newcommand{\CC}{\mathbb{C}}
\newcommand{\QQ}{\mathbb{Q}}
\newcommand{\RR}{\mathbb{R}}
\newcommand{\ZZ}{\mathbb{Z}}
\newcommand{\NN}{\mathbb{N}}

\newcommand{\WW}{\mathbb{W}}

\begin{document}

\title[Equivariant $K$-theory and cobordism of GKM spaces]{Integral equivariant $K$-theory and cobordism ring of simplicial GKM orbifold complexes}

\author[K Brahma]{Koushik Brahma}
\address{Department of Mathematics, Indian Institute of Technology Madras, India}
\email{koushikbrahma95@gmail.com}

\author[S. Sarkar]{Soumen Sarkar}
\address{Department of Mathematics, Indian Institute of Technology Madras, India}
\email{soumen@iitm.ac.in}



\date{\today}
\dedicatory{}

\abstract In this paper, we define `simplicial GKM orbifold complexes' and study some of their topological properties. We introduce the concept of filtration of regular graphs and `simplicial graph complexes', which have close relations with simplicial GKM orbifold complexes. We discuss the necessary conditions to confirm an invariant $q$-CW complex structure on a simplicial GKM orbifold complex. We introduce `buildable' and `divisive' simplicial GKM orbifold complexes. We show that a buildable simplicial GKM orbifold complex is equivariantly formal, and a divisive simplicial GKM orbifold complex is integrally equivariantly formal. We give a combinatorial description of the integral equivariant cohomology ring of certain simplicial GKM orbifold complexes. We prove the Thom isomorphism theorem for orbifold $G$-vector bundles for equivariant cohomology and equivariant $K$-theory with rational coefficients. We extend the main result of Harada-Henriques-Holm (2005) to the category of $G$-spaces equipped with `singular invariant stratification'. 
We compute the integral equivariant cohomology ring, equivariant $K$-theory ring and equivariant cobordism ring of divisive simplicial GKM orbifold complexes. We describe a basis of the integral generalized equivariant cohomology of a divisive simplicial GKM orbifold complex. 
\endabstract

\maketitle


\section{Introduction}
Group actions on topological spaces provide rich structures of some classes of manifolds. Examples of such actions include torus actions on symplectic manifolds \cite{Aud, Dus}, locally standard torus actions on even-dimensional manifolds \cite{DJ}, torus action on GKM manifolds \cite{GKM98}. The above references mainly consider manifolds with torus actions. Moreover, there are `nice' topological spaces with a `similar type' of torus actions which may not be manifold. Let $M_1$, $M_2$ be two manifolds in one of the above categories and $f_1 \colon A \hookrightarrow M_1$ and $f_2 \colon A \hookrightarrow M_2$ be two equivariant embeddings. Then one can construct the pushout $(M_1\cup_{A}M_2)$ by the following commutative diagram.
\[ \begin{tikzcd}
	A \arrow{r}{f_1} \arrow{d}{f_2} & M_1 \arrow{d} \\%
	M_2 \arrow{r} & (M_1\cup_{A}M_2).
\end{tikzcd}
\]

The space $(M_1\cup_{A}M_2)$ is not a manifold in general. Of course, one can make this definition in any category. However, we restrict this definition to the category of GKM orbifolds. Observing this construction and the concept of simplicial complexes, we introduce `simplicial GKM orbifold complexes'. The word GKM originated after the pioneering work of \cite{GKM98}, which describes the $T$-equivariant cohomology ring $H_T^{*}$ of any smooth projective toric variety in terms of combinatorial data obtained from its 1-dimensional (complex) orbit structure. Here $T$ is a compact abelian torus. Subsequently, several papers extended their result to `nice' stratified spaces see, for example, \cite{GZ01}, \cite{HHH}, and \cite{DKS}. We note that the class of simplicial GKM orbifold complexes contain the weighted projective spaces \cite{Ka}, the toric orbifolds \cite[Section 7]{DJ}, the projective toric varieties \cite{SaSo} and the weighted Grassmann orbifolds \cite{BS_Gra}.

Let $G$ be a topological group and $X$ a $G$-space. The fixed point set of the $G$-action on $X$ is defined by $X^G:=\{x \in X ~|~ gx=x ~\mbox{for all} ~ g \in G\}$. If $Y \subseteq X$ is an invariant subset, then the subgroup $G_Y= \{g \in G ~|~ gy=y; \text{ for all } y\in Y\}$ is called the isotropy group of $Y$. For $x \in X$, the set $G{\cdot}x=\{gx \in X ~|~ g \in G\}$ is called the orbit of $x$. 

The diagonal $G$-action on $EG \times X$ defined by $g(\mathfrak{e}, x) \mapsto (\mathfrak{e}g^{-1}, g x)$ is free, and the orbit space $(EG \times X)/G$ has a cell-structure if $X$ has so. The Borel equivariant cohomology of the $G$-space $X$ is defined by the cohomology of $X_G:= (EG \times X)/G$. That is $$H^{\ast}_G(X;R) :=H^{\ast}(X_G;R),$$
where $R$ is a commutative ring. The free action of $G$ on $EG$ and the $G$-equivariant projection $EG \times X \to EG$ induce the following fiber bundle: 
$$X \xrightarrow{\iota} X_G \xrightarrow{\pi} BG.$$ Therefore, one has $$H^{\ast}(BG;\QQ) \xrightarrow{\pi^{*}} H^{\ast}(X_G;\QQ) \xrightarrow{\iota^{*}} H^{\ast}(X;\QQ).$$
If $\iota^*$ is a surjective map, then the $G$-space $X$ is called an equivariantly formal space. One can also define integrally equivariantly formal space $X$ if the above condition is true with integer coefficients. Note that if $H^{odd}(X;\ZZ)=0$ and $H^*(X;\ZZ)$ is torsion free then $X$ is integrally equivariantly formal. The readers are referred to \cite{May} for details and results on $G$-equivariant cohomology.

Let $X$ be a compact $G$-space. The set of all isomorphism classes of complex $G$-vector bundles on $X$ is an abelian semigroup under the direct sum operation. Let $K_G(X)$ be the associated Grothendieck group. The tensor product of $G$-vector bundles induces a multiplication structure on $K_G(X)$. Then $K_G(X)$ is a commutative ring with a unit. If $X$ is a point $\{pt\}$, then $K_G(\{pt\})=R(G)$, the representation ring of $G$, {see \cite{Hus}}. 

Let $X^{+}$ be the one-point compactification of a locally compact $G$-space $X$. 
 If $X$ is already a compact space then $X^{+}=X \sqcup \{pt\}$. Denote $K_G^0(X) :=\widetilde{K}_G(X^{+})$. For any $n\in \NN$, one defines 
$$K_G^{-n}(X) :=\widetilde{K}_G(S^n X^{+}),$$ see \cite{Seg}.
Thus one can obtain a cohomology theory of Eilenberg-Steenrod \cite{EiSt} known as the equivariant $K$-theory and denoted by $K_G^{*}(X)$ for a locally compact $G$-space $X$. If $X$ is a point then $K_G^{*}(\{pt\})=R(G)[z,z^{-1}]$, where $z$ denotes the Bott periodicity element having cohomological dimension $-2$.

The $G$-equivariant ring $MU^*_G(X)$ is known as the equivariant complex cobordism ring, see \cite{TD}. Sinha \cite{Sinha} and Hanke \cite{Han} have shown several developments on $MU^*_G$. However, many interesting questions on $MU^*_G(X)$ remain undetermined. For example, $MU^*_G(\{pt\})$ is only partially known for a nontrivial group $G$.

The equivariant map $X \rightarrow \{pt\}$ induces a graded $\mathcal{E}_{G}^*(\{pt\})$-algebra structure on $\mathcal{E}_{G}^*(X)$ where $\mathcal{E}_{G}^*$ represents the generalized $G$-equivariant cohomology theory. 
 In \cite{HHH}, the authors have calculated the generalized $G$-equivariant cohomology ring $\mathcal{E}_G^{*}(X)$ and a module generator of $\mathcal{E}_G^{*}(X)$ of a class of equivariant stratified $G$-spaces $X$. So, the following question arises. How to calculate the generalized $G$-equivariant cohomology ring of a GKM orbifold and, more generally, of a GKM orbifold simplicial complex? This question is partially answered in \cite{SaSo} for toric varieties associated with almost simple lattice polytopes. In this paper, we study the integral generalized $G$-equivariant cohomology theory $\mathcal{E}_G^{*}$ (such as $H^*_G$, $K^*_G$, and $MU^*_G$) of simplicial GKM orbifold complexes under some hypothesis arising from the combinatorial structure of its 1-dimensional (complex) invariant subsets called spindles in \cite{GZ01}.

The paper is organized as follows.
 In Section \ref{sec_simp_gkm_comp}, we recall regular graphs and define `simplicial graph complexes'. Then, we introduce the concept of filtration of regular graphs and simplicial graph complexes. We prove that if $\Gamma$ is a connected regular graph, then there always exists a filtration of $\Gamma$, see Proposition \ref{prop:filt_reg_gh}.  

We rewrite the definition of GKM orbifolds. We describe a process of constructing certain invariant equivariantly contractible subsets of a GKM manifold. If these subsets induce a $T$-invariant stratification of a GKM manifold, then we call the GKM manifold buildable, see Definition \ref{def_bul_gkm_mfd}. We give several examples of buildable GKM manifolds. We prove that the integral cohomology of a buildable GKM manifold has no torsion and is concentrated in even degrees, see Theorem \ref{thm_eq_formal}. 
Then, we define `simplicial GKM orbifold complexes' in Definition \ref{def_gkm_simp_orbifold} extending the concept of simplicial complexes and GKM orbifolds. We discuss several examples of simplicial GKM orbifold complexes which are not GKM orbifolds. We construct a $T$-invariant stratification of a simplicial GKM orbifold complex. We introduce a `buildable simplicial GKM orbifold complex' and a `divisive simplicial GKM orbifold complex', see Definition \ref{def_filt_simp_gkm}. 
We show that a buildable simplicial GKM orbifold complex is equivariantly formal, and a divisive simplicial GKM orbifold complex is integrally equivariantly formal.

 We define a simplicial GKM graph complex and its equivariant cohomology extending the concept of the GKM graph and its equivariant cohomology following \cite[Section 1.7]{GZ01}. We give an example of a non-trivial simplicial GKM graph complex. We describe how one can obtain a simplicial GKM graph complex $\mathcal{G}_{\mathcal{K}}$ from a simplicial GKM orbifold complex $\mathcal{K}$. Under some hypotheses, we prove that the equivariant cohomology ring $H_{T}^{*}(\mathcal{K};\mathbb{Z})$ is isomorphic to $H_{T}^{*}(\mathcal{G}_{\mathcal{K}}, \alpha, \theta)$ as $H_{T}^{*}(\{pt\};\mathbb{Z})$ algebra, where $H_{T}^{*}(\mathcal{G}_{\mathcal{K}}, \alpha, \theta)$ is the equivariant cohomology of the simplicial GKM graph complex $\mathcal{G}_{\mathcal{K}}$. We note that this result generalizes \cite[Theorem 2.9]{DKS}.

In Section \ref{sec_thom_iso}, we briefly discuss orbifold $G$-vector bundles for a topological group $G$. Then, we prove the Thom isomorphism for equivariant cohomology and equivariant $K$-theory for an orbifold $G$-vector bundle, see Theorem \ref{prop_thm_global}. This result generalizes the equivariant Thom isomorphism studied in \cite[Section 2]{SaSo}. 
 
 In Section \ref{sec_gen_cohom}, we extend the main result \cite[Theorem 2.3]{HHH} to the category of $G$-spaces equipped with a stratification $$\{pt\} = X_0 \subset X_1 \subset X_2  \subset \cdots, $$ where $X = \cup_{j \geq 0} X_j$ and $X_j\setminus X_{j-1}$ is equivariantly homeomorphic to an orbifold $G$-vector bundle for $j\geq 1 $, see Proposition \ref{prop_gen_cohom}. Under a hypothesis on this stratification, we can give a presentation of the equivariant cohomology and the equivariant $K$-theory ring of $X$, see Proposition \ref{prop:equivariant_k-theory}. We describe the equivariant cohomology and equivariant $K$-theory ring of a buildable simplicial GKM orbifold complex with rational coefficients, see Theorem \ref{thm:buildable_gkm}. For `divisive' simplicial GKM orbifold complexes, we compute their equivariant cohomology, equivariant $K$-theory, and equivariant cobordism ring with integer coefficients, see Theorem \ref{thm:buildable_gkm_2}. In Example \ref{ex:equi_cohom_gkmplx}, we explicitly describe the ring $\mathcal{E}_T^*(\mathcal{K};\ZZ)$ for a divisive simplicial GKM orbifold complex $\mathcal{K}$. We discuss some $\mathcal{E}_T^*(\{pt\};\ZZ)$-module generators of these equivariant cohomologies of divisive simplicial GKM orbifold complexes, see Lemma \ref{lem_module1} and Theorem \ref{thm_module2}.

\section{Graphs, filtrations and Simplicial GKM orbifold complexes}\label{sec_simp_gkm_comp}

In this section, we introduce simplicial graph complexes and simplicial GKM orbifold complexes following the concept of simplicial complexes. We construct a filtration of a regular graph and talk about the filtration of a simplicial graph complex. This filtration and some differential geometric properties may induce an invariant stratification on the simplicial GKM orbifold complexes. This facilitates the introduction of buildable and divisive simplicial GKM orbifold complexes. Then we prove that buildable simplicial GKM orbifold complexes are equivariantly formal and divisive simplicial GKM orbifold complexes are integrally equivariantly formal.
We define the equivariant cohomology of simplicial  GKM graph complexes and compare it with the equivariant cohomology of simplicial GKM orbifold complexes.

\subsection{Filtration of a simplicial graph complex}
 In this subsection, we define simplicial graph complexes and their filtrations. Then we study some properties of these filtrations. 
 
 Let $V$ be a non-empty set and $E$ a non-empty subset of $V \times V$ such that an element $e\in E$ can be written as $e=(x,y)$ for some $x,y\in V$ and $x\neq y$. Then the pair $(V, E)$ is called a graph. 
The element $e $ is called the edge incident to the vertices $x$ and $y$, and the vertices $x$ and $y$ are called adjacent vertices.  
The degree of a vertex is defined by the number of edges that are incident to the vertex. A graph is called a regular graph of degree $n$ if the degree of all vertices is $n$. For example, the vertices and the edges of an $n$-dimensional simple polytope form a regular graph of degree $n$.

We construct a filtration of a connected regular graph. Let $\Gamma=(V, E)$ be a connected regular graph of degree $n$, where $V$ is the vertices and $E$ is the edges of $\Gamma$. Let ${b_0\in V}$, $V_0=\{b_0\}$ and $\Gamma_0:=(V_0,E_0)$ where $E_0=\emptyset$. Next, we consider a vertex $b_1\in V \setminus V_0$ which is adjacent to $b_0$. Let $V_1=\{b_0,b_1\}$ and $E_1$ be the edge joining $b_0$ and $b_1$. Define $\Gamma_1:=(V_1,E_1)$. Suppose, inductively, we define $\Gamma_k:=(V_k,E_k)$ where $V_k=\{b_0,b_1,\dots,b_k\}$ and $E_k$ is the set of edges in $E$ whose vertices are in $V_k$. Let $V \setminus V_k$ be nonempty and $$k':=\min \{\ell\in \{0,1,2,\dots,k\}~|~b_{\ell}~ \mbox{is adjacent to a vertex in}~ V \setminus V_k\}.$$ Since $\Gamma$ is a connected regular graph, there exists a $b_{k+1}\in V\setminus V_k$ satisfying that $b_{k+1}$ is adjacent to $b_{k'}$.
  Let $V_{k+1}:=\{b_0,b_1,\dots,b_k,b_{k+1}\}$ and $E_{k+1}$ be the set of edges in $E$ whose vertices are in $V_{k+1}$. So
\begin{equation}\label{ret_mfd}
 E_{k+1}:=\{e\in E~|~V(e)\subset V_{k+1}\}=E_k\sqcup \{e\in E~|~b_{k+1}\in V(e)\subset V_{k+1}\}.
\end{equation}
Then define $\Gamma_{k+1}:=(V_{k+1},E_{k+1})$.
This process stops when there are no remaining vertices. Therefore
\begin{equation}\label{fil_graph}
	\{pt\}= \Gamma_0\subset \Gamma_1\subset \cdots \subset \Gamma_m= \Gamma 
\end{equation}
 gives a filtration of $\Gamma$, since $\Gamma$ is a connected graph, where $m+1=|V|$. Note that \eqref{fil_graph} induces an ordering on the vertices of the graph $\Gamma$. Therefore, one gets the following. 
 
\begin{proposition}\label{prop:filt_reg_gh}
If $\Gamma$ is a connected regular graph, then there always exists a filtration of $\Gamma$ as in \eqref{fil_graph}.
\end{proposition}

 \begin{figure}
 	\begin{tikzpicture}[scale=.35]

 		\node at (2.5,-4) {$\bullet$};
 		\node at (2.5,-3) {$b_0$};
 	 	\node at (2.5,-7) {$\Gamma_0$};

 		\begin{scope}[xshift=115]
 			\draw (2.5,-5)--(0,0);
 			\node at (2.5,-5) {$\bullet$};
 			\node at (0,0) {$\bullet$};
 		\node at (3,-4) {$b_0$};
 				\node at (1,1) {$b_1$};
 	 	\node at (2.5,-7) {$\Gamma_1$};

 	\end{scope}
 			\begin{scope}[xshift=230]
 			\draw (2.5,-5)--(0,0);
 			\draw (2.5,-5)--(3,-2);
 			\draw (0,0)--(3,-2);
 			\node at (2.5,-5) {$\bullet$};
 			\node at (0,0) {$\bullet$};
 		\node at (3,-2) {$\bullet$};
 		\node at (3.5,-4.5) {$b_0$};
 				\node at (1,1) {$b_1$};
 			\node at (4.2,-1.9) {$b_2$};
 	 	\node at (2.5,-7) {$\Gamma_2$};
 		\end{scope}
 		
 			\begin{scope}[xshift=395]
 				\draw (0,0)--(3,-2);
 			\draw (6,0)--(3,-2);
 		
 			\draw (2.5,-5)--(0,0);
 			\draw (2.5,-5)--(3,-2);
 			\draw (2.5,-5)--(6,0);
 				\node at (2.5,-5) {$\bullet$};
 			\node at (0,0) {$\bullet$};
 		\node at (3,-2) {$\bullet$};
 		 		\node at (6,0) {$\bullet$};
 		 			\node at (3.6,-4.5) {$b_0$};
 				\node at (1,1) {$b_1$};
 			\node at (3.5,-.8) {$b_2$};
 			\node at (7,0) {$b_3$};
 	 	\node at (2.5,-7) {$\Gamma_3$};
 			\end{scope}
 			
 				\begin{scope}[xshift=660]
 				\draw (0,0)--(2,-2);
 			\draw (6,0)--(2,-2);
 			\draw  (0,0)--(4,2);
 			\draw  (6,0)--(4,2);
 			\draw[dashed] (2.5,-5)--(4,2);
 			\draw (2.5,-5)--(0,0);
 			\draw (2.5,-5)--(2,-2);
 			\draw (2.5,-5)--(6,0);
 				\node at (2.5,-5) {$\bullet$};
 			\node at (0,0) {$\bullet$};
 		\node at (2,-2) {$\bullet$};
 		\node at (6,0) {$\bullet$};
 	 \node at (4,2) {$\bullet$};	
 	 	\node at (3.6,-4.5) {$b_0$};
 	\node at (0,1) {$b_1$};
 			\node at (2,-.8) {$b_2$};
 			\node at (7,0) {$b_3$};
 			 \node at (4,3) {$b_4$};
 	 	\node at (2.5,-7) {$\Gamma_4$};
 			\end{scope}
 			
 			\begin{scope}[xshift=935]
 			\draw (0,0)--(2,-2);
 			\draw (6,0)--(2,-2);
 			\draw [dashed] (0,0)--(4,2);
 			\draw [dashed] (6,0)--(4,2);
 			\draw[dashed] (3.5,5)--(4,2);
 			\draw (3.5,5)--(0,0);
 			\draw (3.5,5)--(2,-2);
 			\draw (3.5,5)--(6,0);
 			\draw[dashed] (2.5,-5)--(4,2);
 			\draw (2.5,-5)--(0,0);
 			\draw (2.5,-5)--(2,-2);
 			\draw (2.5,-5)--(6,0);
 				\node at (2.5,-5) {$\bullet$};
 			\node at (0,0) {$\bullet$};
 		\node at (2,-2) {$\bullet$};
 		 		\node at (6,0) {$\bullet$};
 	 		 		\node at (4,2) {$\bullet$};	
 	 		 	\node at (3.5,5) {$\bullet$};
 \node at (3.6,-4.5) {$b_0$};
 				\node at (-0.9,1) {$b_1$};
 			\node at (1.7,-.8) {$b_2$};
 			\node at (7,0) {$b_3$};
 			 \node at (5.2,2.5) {$b_4$};
 			 \node at (3.5,6) {$b_5$};
 	 	\node at (2.5,-7) {$\Gamma_5$};
 			\end{scope}
     \end{tikzpicture}
     
 \caption{A filtration of a regular graph of degree 4.}  
\end{figure}
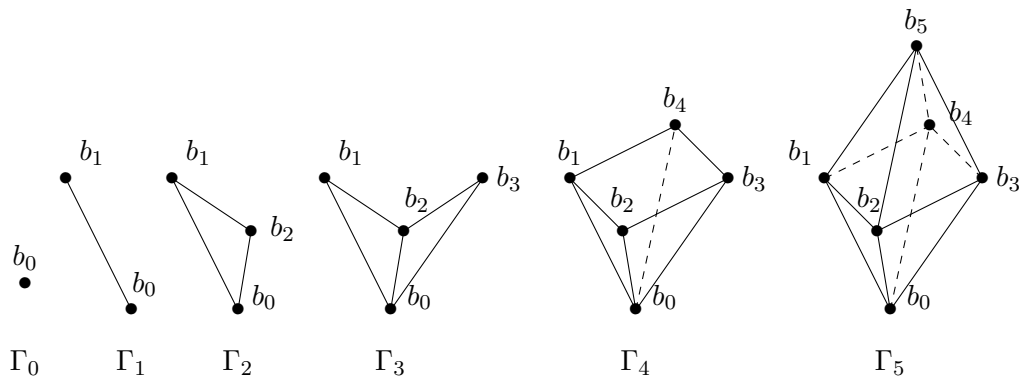

Next, we introduce simplicial graph complexes generalizing the notion of regular graphs. Let $(V, E)$ and $(V', E')$ be two graphs. We define the intersection $(V, E) \cap (V', E')$ is the graph $(V \cap V', E \cap E')$.

\begin{definition}\label{Regular graphical complex}
	Let $\mathcal{G} = \{(V_{\beta}, E_{\beta})\}_{\beta \in \Lambda}$ be a finite collection of regular graphs. We say that $\mathcal{G}$ is a simplicial graph complex if $(V_\beta, E_\beta), (V_\delta, E_\delta) \in \mathcal{G}$ and $(V_\beta, E_\beta) \cap (V_\delta, E_\delta) \neq \emptyset$ then $(V_\beta, E_\beta) \cap (V_\delta, E_\delta) \in \mathcal{G}$ and $(V_\beta, E_\beta) \cap (V_\delta, E_\delta)$ is a regular subgraph of both $(V_\beta, E_\beta)$ and $(V_\delta, E_\delta)$.
\end{definition}

If $\mathcal{G}$ is a simplicial graph complex, then denote  $$V(\mathcal{G}) := \bigcup_{\beta\in{\Lambda}} V_{\beta} ~\mbox{ and }~ E(\mathcal{G}):= \bigcup_{\beta\in{\Lambda}} E_{\beta}.$$
Then $(V(\mathcal{G}), E(\mathcal{G}))$ is a graph. For simplicity, we may denote this graph by $\mathcal{G}$ and call a simplicial graph complex if the structure is clear.

We can generalize the concept of the filtration of regular graphs to the simplicial graph complexes. Let $\mathcal{G} = \{(V_{\beta}, E_{\beta})\}_{\beta \in \Lambda}$ be a simplicial graph complex with the associated graph $(V(\mathcal{G}), E(\mathcal{G}))$. 
In this case, $\mathcal{G}_0$ and $\mathcal{G}_1 $ are defined similarly as in the first two steps in the construction of a filtration of a connected regular graph. Suppose inductively we can define $\mathcal{G}_k=(V_k,E_k)$ for $k\geq 1$. Let $V\setminus V_k$ be nonempty and
 $$k':=\min \{\ell\in \{0,1,2,\dots,k\} ~|~ b_{\ell}~ \mbox{is adjacent to a vertex in}~ V\setminus V_k\}.$$ Now we consider $b_{k+1}\in V\setminus V_k$ satisfying that $b_{k+1}$ is adjacent to $b_{k'}$. Let ${V_{k+1}:=V_k\sqcup\{b_{k+1}\}}$ and 
 \begin{equation}\label{ret_sim_cplx}
 	E_{k+1}:=E_k\sqcup \{e\in E ~|~ b_{k+1}\in V(e)\subset V_{k+1}\cap V_{\beta_k}~ \mbox{and}~ e\in E_{\beta_k}\},
\end{equation}
for a unique regular graph $(V_{\beta_k},E_{\beta_k})$ in $\mathcal{G}$. 
 Therefore, 
 \begin{equation}\label{eq_fil_sim_gp_cplx}
 \{pt\} =\mathcal{G}_0\subset \mathcal{G}_1\subset \cdots \subset \mathcal{G}_m.
 \end{equation}

 \begin{definition}\label{def_smp_gph_cplx}
If $\mathcal{G}_m=\mathcal{G}$ for some $m$ then \eqref{eq_fil_sim_gp_cplx} is called a filtration of the simplicial graph complex $\mathcal{G}$.
 \end{definition}

  Note that $E_{k+1}$ in \eqref{ret_mfd} and \eqref{ret_sim_cplx} may not be the same in general unless the simplicial graph complex contains only one regular graph. Also, a filtration of a simplicial graph complex defines an ordering on the vertices of the simplicial graph complex $\mathcal{G}$.
  
 \begin{example}
 	In Figure \ref{Fig_retraction sequence}, we give a filtration of simplicial graph complex $\mathcal{G}$ which is obtained by identifying an edge of the boundary of a triangle and a rectangle.

 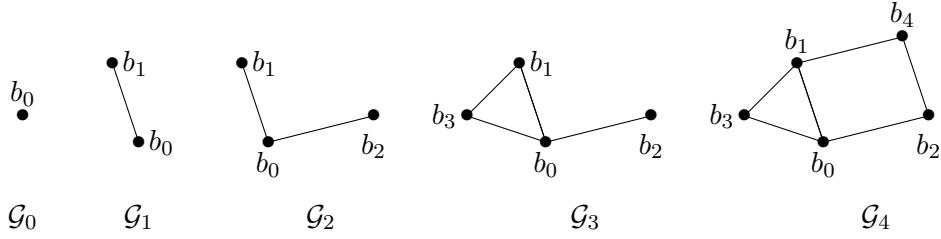
\begin{figure}
 	\begin{tikzpicture}[scale=0.35]
 	
 		  	\node at (0,0) {$\bullet$};
 		  \node[above] at (0,0) {$b_0$};
 		  \node[below] at (0,-3) {$\mathcal{G}_0$};

 		\begin{scope}[xshift=40]
 			
 			\draw (2,2)--(3,-1);
 			\node at (3,-1) {$\bullet$};
 			\node at (2,2) {$\bullet$};
 			\node[right] at (3,-1) {$b_0$};
 			\node[right] at (2,2) {$b_1$};
 			\node[below] at (3,-3) {$\mathcal{G}_1$};
 		
 		\end{scope}
 		
 		\begin{scope}[xshift=180]
 			
 			\draw (2,2)--(3,-1);
 			\draw (3,-1)--(7,0)--cycle;
 			\node at (7,0) {$\bullet$};
 			\node at (2,2) {$\bullet$};
 			\node at (3,-1) {$\bullet$};
 			\node[above] at (7,-2) {$b_2$};
 			\node[below] at (3,-1) {$b_0$};
 			\node[right] at (2,2) {$b_1$};
 			\node[below] at (5,-3) {$\mathcal{G}_2$};

 		\end{scope}
 		
 		\begin{scope}[xshift=480]
 			
 			\draw (2,2)--(3,-1);
 			\draw (3,-1)--(7,0)--cycle;
 			\draw (0,0)--(2,2)--(3,-1)--cycle;
 			\node at (7,0) {$\bullet$};
 			\node at (2,2) {$\bullet$};
 			\node at (0,0) {$\bullet$};
 			\node at (3,-1) {$\bullet$};
 			\node[above] at (7,-2) {$b_2$};
 			\node[below] at (3,-1) {$b_0$};
 			\node[right] at (2,2) {$b_1$};
 			\node[left] at (0,0) {$b_3$};
 				\node[below] at (4.5,-3) {$\mathcal{G}_3$};
 		
 		\end{scope}
 		
 		\begin{scope}[xshift=780]
 		
 		\draw (6,3)--(2,2)--(3,-1)--(7,0)--cycle;
 		\draw (0,0)--(2,2)--(3,-1)--cycle;
 		\node at (7,0) {$\bullet$};
 		\node at (2,2) {$\bullet$};
 		\node at (0,0) {$\bullet$};
 		\node at (6,3) {$\bullet$};
 		\node at (3,-1) {$\bullet$};
 		\node[above] at (7,-2) {$b_2$};
 		\node[below] at (3,-1) {$b_0$};
 		\node[above] at (2,2) {$b_1$};
 		\node[left] at (0,0) {$b_3$};
 		\node[above] at (6,3) {$b_4$};
 		\node[below] at (5,-3) {$\mathcal{G}_4$};	

 		\end{scope}
 		
 	\end{tikzpicture}
 	\caption{A filtration of a simplicial graph complex.}
 	\label{Fig_retraction sequence}
 \end{figure}
 
\end{example}
 
We remark that there are simplicial graph complexes that do not possess any filtration. For example, let $E_1, E_2, E_3$ be the edges of a triangle. Then $\mathcal{G}= \{ (V(E_i), E_i)\}_{i=1}^3$ is a simplicial graph complex. But there does not exist a filtration of $\mathcal{G}$. The problem arises to construct $\mathcal{G}_3$ from $\mathcal{G}_2$.

\subsection{Stratification of a simplicial GKM orbifold complex}
 In this subsection, we define simplicial GKM orbifold complexes and their $T$-invariant stratification. Then we study some topological properties of these spaces and invariant stratification. 
 
 Let $T$ be a compact real torus acting effectively on a $2n$-dimensional orbifold $X$. The set $X_1:= \{x \in X ~|~ \dim(T{\cdot}x) \leq 1\}$ is called the one skeleton of $X$. Let $p \in X^T$ be an isolated fixed point. Then there is an orbifold chart $(\widetilde{U}, \xi, K)$ over a neighborhood $U \subset X$ of $p$ and a finite covering $\widetilde{T}$ of $T$ such that $\widetilde{T}$ acts on $\widetilde{U}$ effectively, the map $\xi \colon \widetilde{U} \to U$ preserves the respective group actions,  $\widetilde{p}$ is the $\widetilde{T}$ fixed point in $\widetilde{U}$ with  $\widetilde{p} = \xi^{-1}(p)$ and $\widetilde{T}/K \cong T$, see \cite[Proposition 2.8]{GGKRW}. Since $T$ is abelian, the tangent space of $\widetilde{U}$ at $\widetilde{p}$ can be decomposed as $$T_{\widetilde{p}}\widetilde{U} \cong \mathcal{V}(\widetilde{\alpha}_1) \oplus \cdots \oplus \mathcal{V}(\widetilde{\alpha}_n)$$
where $\mathcal{V}(\widetilde{\alpha}_i)$ is a one dimensional irreducible representation of $\widetilde{T}$ and $\widetilde{\alpha}_i$ is the corresponding character for $i=1,2,\dots, n$. Let $\alpha_i$ be the image of $\widetilde{\alpha}_i$ under the Lie algebra map ${L(\widetilde{T}) \to L(T)}$ induced by the covering homomorphism $\widetilde{T} \to T$. Then $$T_pU \cong T_{\widetilde{p}}\widetilde{U}/K \cong (\mathcal{V}(\widetilde{\alpha}_1)/K_1) \oplus \cdots \oplus (\mathcal{V}(\widetilde{\alpha}_n)/K_n)$$ for some subgroups $K_1, \ldots, K_n$ of $K$. We say that $\alpha_i$ is the character of the irreducible $T$-representation $\mathcal{V}(\alpha_i) \cong \mathcal{V}(\widetilde{\alpha}_i)/K_i  \subset T_pX$ at $p$ for $i=1, \ldots, n$. We recall the definition of a GKM orbifold following \cite{GZ01, DKS}.

\begin{definition}\label{def_gkm_orbifold}
	A $2n$-dimensional compact $T$-orbifold $X$ is said to be a GKM orbifold if the following holds.
\begin{enumerate}
\item $X^T$ is finite and there are finitely many connected components of $X_1$. 
\item For $p \in X^T$, let $T_pX \cong  \mathcal{V}({\alpha}_1) \oplus \cdots \oplus \mathcal{V}(\alpha_n)$ for some orbifold line bundles $\mathcal{V}({\alpha}_1), \ldots, \mathcal{V}(\alpha_n)$. Then the vectors $\alpha_1, \ldots, \alpha_n$ are pairwise linearly independent. 
\end{enumerate}
\end{definition} 

\begin{example}
Let $X$ be a spindle $\WW P(q_0,q_1)$ defined by $$\WW P(q_0,q_1):=\CC^2\setminus \{0\}/\sim,$$
for some natural numbers $q_0, q_1$, where $(z_0:z_1)\sim (t^{q_0}z_0:t^{q_1}z_1)$ for some $t\in \CC^{*}$.  There exists a $S^1$-action on $\WW P(q_0,q_1)$ defined by $t\cdot[z_0:z_1]_{\sim}=[z_0:tz_1]_{\sim},$ for $t\in S^1$. Then $X$ is a 2-dimensional GKM orbifold sometimes called by `orbifold $S^2$'. Note that $X$ has two fixed points $e_0=[1:0]_\sim$ and $e_1=[0:1]_\sim$.
\end{example}

\begin{remark}\label{rem:spindle} 
If $X$ is a manifold that satisfies the assumption of Definition \ref{def_gkm_orbifold} then $X$ is called a GKM manifold. For a GKM orbifold $X$, the set $X_1$ is a  finite union of spindles.
 Let $V= X^T$ and $E$ be the set of spindles. Then $(V, E)$ has the structure of a regular graph of degree $n$, see \cite{GZ01}. We say that $(V, E)$ is the graph of the GKM orbifold $X$. 
\end{remark}

Let $R$ be a Riemannian metric on a GKM manifold $M$.  Let $B(0_p,r)$ denote the open ball in the tangent space $T_pM$ defined by $$B(0_p,r)=\{v\in T_pM ~|~ ||v||=R_p(0_p,v)< r\}$$ where $0_p$ is the zero vector in $T_pM$. Then for each $v \in B(0_p,r)$ there is a unique geodesic $\sigma_v(t)$ with $\sigma_v(0) = p$ and $\sigma'_v(0) = v$, where $\sigma_v(t)$ is defined for $|t| \leq 1$.
 The exponential map at $p$ is the map
$$\exp_p \colon B(0_p,r) \to M \mbox{  defined by  } \exp_p(v) = \sigma_v(1).$$ \cite[Proposition 4.4.4]{Muk} showed that the exponential map $\exp_p$ is a diffeomorphism on an open neighborhood of $0_p$ in $T_pM$ to an open neighborhood of $p$ in $M$, and it is defined explicitly by $\sigma_v(t)=\exp_p(tv)$.

Let $p$ be a fixed point of the $T$-action on $M$. If $g\in T$ then $g \colon M \to M$ is a diffeomorphism which induces an isomorphism $dg_p \colon T_pM \to T_pM$ such that the following diagram commutes.
\begin{equation}\label{eq_com}
     \begin{tikzcd}
	T_pM\arrow{r}{dg_p} \arrow{d}{\exp_p} & T_pM \arrow{d}{\exp_p} \\%
	M \arrow{r}{g} & M.
\end{tikzcd}
\end{equation}

 Let $\Gamma=(V,E)$ be the graph of a GKM manifold $M$ of dimension $2n$. Then $\Gamma$ is a regular graph of degree $n$ and $V$ corresponds bijectively to the set $M^T$ and $E$ correspond bijectively to the set of all invariant spheres which contains two fixed points. Thus, by Proposition \ref{prop:filt_reg_gh}, we can define a filtration of the graph $\Gamma$ of a GKM manifold $M$ and consequently an ordering on the set $M^T$ of all fixed points of $M$. Consider the orbit map $h \colon M_1 \to \Gamma$ where $M_1$ is the one skeleton of the $T$-action on $M$. The map $h$ sends the set $M^T$ bijectively to $V$ and the $T$-invariant spheres of $M$ bijectively to $E$ such that if $S^2_i$ is an invariant sphere with fixed points $N_i$ and $S_i$ then $h(N_i)$ and $h(S_i)$ are the vertices of the edge $h(S^2_i)$. 
 

 Let $p$ be a fixed point of the $T$-action on $M$ and
 $$T_pM=\mathcal{V}(\alpha_1)\oplus \mathcal{V}(\alpha_2)\oplus\dots \oplus \mathcal{V}(\alpha_n) $$ the decomposition of $T_pM$ into one dimensional irreducible representations.
 Then using \eqref{eq_com}, $\exp_p(\mathcal{V}(\alpha_j))$ is a submanifold of an invariant sphere $S_{j}^2$ in $M$ containing $p$ for $j=1,2,\dots,n$.
 Recall the filtration  of $\Gamma=(V, E)$ as in \eqref{fil_graph} gives an ordering say $\{b_0, \ldots, b_m\}$ on $V$. Then $h(p)=b_i\in V$ for a unique $i \in  \{0, \ldots,  m\}$. 
Let $F_i=E_i \setminus E_{i-1}$ and
$\{e_{i_j}\}_{j=1}^{d_i}$
be the edges in $F_i$ containing $b_i$. Let $h^{-1}(e_{i_j})=S^2_{i_j} \setminus \{b_j\}$, for $j=1,\dots,d_i$.
Let $\mathcal{V}(\alpha_{i_j})$ be the representation corresponding to the $T$-invariant $S^2_{i_j} \setminus \{b_j\}$, for $\alpha_{i_j}\in \{\alpha_1,\dots,\alpha_n\}$ and $j\in \{1,\dots,d_i\}$.
Then $\oplus_{j=1}^{d_i}\mathcal{V}(\alpha_{i_j})$ is an invariant susbspace of $T_pM$. Let $M_i$ be $T$-invariant submanifold of $M$ of dimension $2d_i$ such that $M_i$ is equivariantly diffeomorphic to a $T$-invariant open neighborhood of the origin in $\oplus_{j=1}^{d_i}\mathcal{V}(\alpha_{i_j})$ under the exponential map.


\begin{lemma}
 The submanifold $M_i$ is equivariantly contractible to $h^{-1}(b_i)=p$.
\end{lemma}
\begin{proof}
This follows from the fact that the complex vector space structure of  $\oplus_{j=1}^{d_i}\mathcal{V}(\alpha_{i_j})$ is equivariantly contractible to the origin $0_p$ and $M_i$ is the equivariant image of an invariant open neighborhood of $0_p$.  
\end{proof}

We call this $(M_i,p)$ a geodesic cell. We consider the subset $M_i \subseteq M$ which is maximal with this property. Let $Y_0 := M_0 = h^{-1}(b_0) \cong \{pt\}$ and $$Y_j:=\bigcup_{i=0}^{j} M_i\subset M$$ such that $\overline{M_j} \setminus M_j \subseteq Y_{j-1}$ for $j=1, \ldots, m=|V|$. The above observation leads to the following.

\begin{definition}\label{def_bul_gkm_mfd}
	A GKM manifold $M$ equipped with the $T$-action is called buildable if, for the above $Y_j$'s and $M_j$'s, the following 
	\begin{equation}\label{eq_g-stratification}
\{pt\}=Y_0 \subset	Y_1 \subset Y_2 \subset \cdots \subset Y_{m}=M
	\end{equation}
    is a $T$-invariant stratification for some $m\geq 1$.
\end{definition}
\begin{example}
From the proof of \cite[Theorem 3.1]{DJ} and \cite[Section 3]{BS_Gra}, one can show that toric manifolds and Grassmann manifolds are buildable GKM manifolds.     
\end{example}

\begin{proposition}\cite[Theorem 4.3]{Gon}
Let $M$ be a normal projective variety with a torus action such that $M$ is a GKM manifold. Then $M$ is buildable. 
\end{proposition}

\begin{remark}
The set $h^{-1}(E_j \setminus E_{j-1})$ is the one skeleton of $Y_j \setminus Y_{j-1}$ for all ${j\in \{1,2,\dots,m\}}$.
\end{remark}

\begin{lemma}\label{prop_geo_cell}
Let $M$ be a buildable GKM manifold with a stratification as in \eqref{eq_g-stratification}. If $M_i$ is a geodesic cell of dimension $2d_i$, then 
\begin{align*}
    H^j(Y_i,Y_{i-1};\ZZ) = \begin{cases}0  \quad &\text{ if } j\neq 2d_i\\
    \ZZ  \quad &\text{ if } j= 2d_i
    \end{cases}
\end{align*}
\end{lemma}
\begin{proof}
The submanifold $M_i$ is homeomorphic to an open ball in $\RR^{2d_i}$. Thus $\partial M_i$ is homeomorphic to $S^{2d_i-1}$. The attaching map induces the cofibration $S^{2d_i-1}\to Y_{i-1}\to Y_i= Y_{i-1} \cup M_i$. Using this cofibration  $$H^j(Y_i,Y_{i-1})\cong H^j(C(S^{2d_i-1}),S^{2d_i-1})\cong H^{j-1}(S^{2d_i-1}).$$ So we get the result.
\end{proof}

\begin{theorem}\label{thm_eq_formal}
The integral cohomology of $Y_i$ in \eqref{eq_g-stratification} has no torsions and vanishes in odd degrees for $i=0,1,\dots,m$. In particular, $M$ is integrally equivariantly formal.
\end{theorem}
\begin{proof}
We prove this by induction on $i$. Consider the cohomology exact sequence of the pair $(Y_i,Y_{i-1})$,
$$\to H^{j-1}(Y_{i-1})\to H^{j}(Y_i,Y_{i-1})\to H^{j}(Y_{i})\to H^{j}(Y_{i-1})\to H^{j+1}(Y_i,Y_{i-1})\to.$$
Using Lemma \ref{prop_geo_cell} $H^{j}(Y_{i})\cong H^{j}(Y_{i-1})$, for $j\neq 2d_i,2d_i-1$. $H^{2d_i-1}(Y_{i-1})=0$ by induction. Thus $H^{2d_i-1}(Y_{i})=0$. Also $H^{2d_i}(Y_i)=H^{2d_i}(Y_{i-1})\oplus \ZZ.$ Therefore, by induction, one can complete the proof. 
\end{proof}

We note that there exists a GKM manifold that is not equivariantly formal; see \cite[Theorem 3.1.1]{GZ01} and the remark after that. Therefore, this GKM manifold is not buildable, though the associated one skeleton has the structure of a regular graph. This justifies the introduction of Definition \ref{def_bul_gkm_mfd}.

Now, we introduce the concept of simplicial GKM orbifold complexes generalizing the definition of simplicial complexes.
\begin{definition}\label{def_simp_orbifold}
Let $\mathcal{K}$ be a finite collection of effective orbifolds. We say that $\mathcal{K}$ is a simplicial orbifold complex if $X, Y \in \mathcal{K}$  and $X \cap Y \neq \emptyset $ then $X \cap Y \in \mathcal{K}$ and $X \cap Y$  is a suborbifold of both $X$ and $Y$.
\end{definition} 
We note that one can think $\mathcal{K}$ as a category where objects are elements in $\mathcal{K}$ and for any two elements $X, Y \in \mathcal{K}$, $\mbox{Mor}(X, Y)$ is either inclusion or empty. The limit of this category, denoted by $|\mathcal{K}|$, is called its geometric realization.
We may not distinguish these two notations if both $\mathcal{K}$ and $|\mathcal{K}|$ are clear from the context.
As usual,  the dimension of a simplicial orbifold complex is defined to be the maximum dimension of the orbifolds in that simplicial orbifold complex. Note that each simplex has an orbifold structure. If all elements in  $\mathcal{K}$ are simplexes, then  $\mathcal{K}$ is called a simplicial complex. Next, we introduce another definition that generalizes GKM-orbifolds.

 \begin{definition}\label{def_gkm_simp_orbifold}
 Let $\mathcal{K}$ be a simplicial orbifold complex and there exists an effective $T$-action on $|\mathcal{K}|$ such that if $X\in \mathcal{K}$ then $X$ is a GKM orbifold with respect to $T/T_X$-action, where $T_X$ denotes the isotropy group of $X$. Then $\mathcal{K}$ is called simplicial GKM orbifold complex.
\end{definition}
\begin{example}
All GKM orbifolds and GKM manifolds are simplicial GKM orbifold complexes.  
\end{example}

\begin{example}
Let $X$ be a $2n$-dimensional quasitoric orbifold as defined in \cite{PS} and $\pi \colon X \to Q$ the corresponding orbit map. From \cite{PS} and Definition \ref{def_gkm_orbifold}, we get that any quasitoric orbifold has a GKM orbifold structure. If $F$ is a $k$-dimensional face then  $\pi^{-1}(F)$ is a $2k$-dimensional quasitoric orbifold, see \cite[Subsection 2.3]{PS}. Let $Q' \subset Q$ be such that $Q'$ is the union of some faces in $Q$. Then $\pi^{-1}(Q')\subset X$ is an example of a simplicial GKM orbifold complex. This gives many examples of simplicial GKM orbifold complexes which are not GKM orbifolds.  
\end{example}

\begin{example}
	Let $\mathcal{P}_1$ and $\mathcal{P}_2$ be weighted Grassmann orbifolds, see \cite{AbMa1, BS_Gra} and $\mathcal{P}_0$ be a proper suborbifold of both $\mathcal{P}_1$ and $\mathcal{P}_2$. Then the pushout $X$ in the following diagram is a simplicial GKM orbifold complex, which is not a GKM orbifold. 
	\[ \begin{tikzcd}
		\mathcal{P}_0 \arrow{r} \arrow{d} & \mathcal{P}_1 \arrow{d} \\%
		\mathcal{P}_2 \arrow{r} & X.
	\end{tikzcd}
	\]
\end{example}


Let $\mathcal{K}$ be a simplicial GKM orbifold complex and $|\mathcal{K}|_1 := \{x \in |K| ~|~ \dim{T{\cdot}x} \leq 1\}$ where $T$-action on $|\mathcal{K}|$ satisfies Definition \ref{def_gkm_simp_orbifold}. Then by the similar argument as in Remark \ref{rem:spindle}, the set $|\mathcal{K}|_1$ has a structure of a simplicial graph complex $\mathcal{G}=(V, E)$. Suppose there is a filtration 
$$(V_0, E_0) \subset (V_1, E_1) \subset \cdots \subset (V_m, E_m) =(V, E)$$ 
of the simplicial graph complex $\mathcal{G}=(V, E)$ as in Definition \ref{def_smp_gph_cplx}. This filtration induces an ordering $\{b_0, \ldots, b_m\}$ on the vertices $V$ of the simplicial graph complex $\mathcal{G}$.

Let $p$ be a fixed point of the $T$-action on $\mathcal{K}$ and $h \colon |\mathcal{K}|_1 \to \mathcal{G}$ the orbit map such that $h(p)=b_i$, where $b_i\in V$ (the set of vertices of $\mathcal{G}$) with $i\geq 1$. Let $X(i)$ be the GKM orbifold which corresponds to the regular graph $(V_{\beta_{i-1}},E_{\beta_{i-1}})$, which is defined in \eqref{ret_sim_cplx}. A finite covering $\widetilde{T}$ of $T$ acts on the orbifold chart $(\widetilde{U},\xi,H)$ over a neighbourhood $U \subseteq X(i)$ of $p$ such that $\xi\colon \widetilde{U}\to U$ preserves the respective group action and $\widetilde{p}= \xi^{-1}(p)$ be the $\widetilde{T}$ fixed point in $\widetilde{U}$. Let
$$T_{\widetilde{p}}\widetilde{U} \cong \mathcal{V}(\widetilde{\alpha}_1) \oplus \cdots \oplus \mathcal{V}(\widetilde{\alpha}_n).$$
 
 
 Let $\{e_{i_j}\}_{j=1}^{d_i}$ be the edges in $F_i=E_i \setminus E_{i-1}$ and $\mathcal{V}(\widetilde{\alpha}_{i_j})$ the representation corresponding to the $T$-action on $h^{-1}(e_{i_j})\cong S^2_{i_j}\setminus \{b_j\}$. Therefore, there exists a $\widetilde{T}$-invariant submanifold $\widetilde{M}_i$ of $\widetilde{U}$ containing $\widetilde{p}$ and $T_{\widetilde{p}}\widetilde{M}_i=\oplus_{j=1}^{d_i}\mathcal{V}(\widetilde{\alpha}_{i_j})$. We denote $\xi(\widetilde{M}_i)$ by $M_i$. Define ${H_i :=\{h  \in H ~|~ h \widetilde{M}_i = \widetilde{M}_i\}}$ and 
\begin{equation}\label{eq_local_grp}
L_i := H/H_i.
\end{equation} 
Then $M_i$ is a $T$-invariant quotient orbifold whose orbifold chart is given by $(\widetilde{M}_i, \xi, L_i)$. So there exists a $T$-invariant suborbifold $M_i$ containing $h^{-1}(b_i)$ and its tangential representation is determined by the characters $\{\alpha_{i_j}\}_{j=1}^{d_i}$. Then the subset $M_i$ is $T$-equivariantly homeomorphic to $\CC^{d_i}/L_i$. Define $M_0 = h^{-1}(b_0) \cong \{pt\}$ and $$\mathcal{K}_j :=\cup_{i=0}^{j} M_i\subset \mathcal{K}$$ for $j=0, 1, \ldots, m=|V|$. The above observation inspires to define the following.

\begin{definition}\label{def_filt_simp_gkm}
 A simplicial GKM orbifold complex $\mathcal{K}$ 
 is called {\it buildable} if there is a filtration $\{pt\}=\mathcal{G}_0\subset \mathcal{G}_1\subset \cdots \subset \mathcal{G}_m=\mathcal{G}$  of its simplicial graph complex $\mathcal{G}$ and a $T$-invariant stratification 
\begin{equation}\label{eq_G-tratification}
\{pt\}=\mathcal{K}_0 \subset \mathcal{K}_1 \subset \mathcal{K}_2 \subset \cdots \subset \mathcal{K}_{m} =\mathcal{K}
\end{equation}
such that $\mathcal{K}_j$ is closed, $h^{-1}(E_j\setminus E_{j-1})$ is the one skeleton of $ \mathcal{K}_j\setminus \mathcal{K}_{j-1}$ and  $\mathcal{K}_j\setminus \mathcal{K}_{j-1}$ is $T$-equivariantly homeomorphic to $\CC^{d_j}/L_j$ for some finite group $L_j$ for $j=1,2,\dots,m$.
	
	In addition, if the group $L_j$ defined in \eqref{eq_local_grp} is trivial for any $j$, then  $\mathcal{K}$ is called a {\it divisive} simplicial GKM orbifold complex. 
\end{definition}

\begin{example}
The toric varieties over almost simple lattice polytopes \cite[Section 3]{SaSo} and the weighted Grassmann orbifolds \cite{AbMa1} are buildable simplicial GKM orbifold complexes. The divisive weighted projective spaces \cite{HHRW}, the retractable toric orbifolds \cite[Section 2]{SU}, the divisive toric varieties \cite[Section 5]{SaSo} and the divisive weighted Grassmann orbifolds \cite{BS_Gra} are divisive simplicial GKM orbifold complexes.
\end{example}

\begin{lemma}
Let $\mathfrak{p}$ be a prime and $\gcd(\mathfrak{p},|L_i|)=1$ for $1\leq i\leq j$.  Then the integral cohomology of $\mathcal{K}_j$ in \eqref{eq_G-tratification} has no $\mathfrak{p}$-torsion.
\end{lemma} 

\begin{proof}
    We have the cofibration ${S^{2d_i-1}}/{L_i}\to \mathcal{K}_{i-1}\to \mathcal{K}_i= \mathcal{K}_{i-1} \cup \frac{\CC^{d_i}}{L_i}$, for every $1\leq i\leq j$. Then the result follows using \cite[Theorem 1.1]{BSS}.
\end{proof}

\begin{theorem}
A buildable simplicial GKM orbifold complex is equivariantly formal and a divisive simplicial GKM orbifold complex is integrally equivariantly formal.
\end{theorem}

\begin{proof}
Let $\mathcal{K}$ be a buildable simplicial GKM orbifold complex. So there exists a filtration as in \eqref{eq_G-tratification}  such that $\mathcal{K}_j\setminus\mathcal{K}_{j-1}=\CC^{d_j}/L_j$ for some $d_j\geq 0$ and $j=1,2,\dots,m$. Therefore, $H^{*}(\mathcal{K};\QQ)$ is concentrated in even degrees.

Moreover, if $\mathcal{K}$ is a divisive then $\mathcal{K}_j\setminus\mathcal{K}_{j-1}=\CC^{d_j}$. Thus, $H^{*}(\mathcal{K};\ZZ)$ is torsion-free and concentrated in even degrees.  
\end{proof}

\begin{remark}
A buildable simplicial GKM manifold complex can be viewed as a divisive simplicial GKM orbifold complex.
Every simplicial GKM orbifold complex may not be buildable. For example,
 $$\mathcal{K}=\{[z_1 : z_2: z_3] \in \CC P^2 ~|~ \mbox{at least one} ~z_i~ \mbox{is zero}\}$$
  is a simplicial GKM orbifold complex. However, this simplicial GKM orbifold complex is not buildable in the sense of Definition \ref{def_filt_simp_gkm}. 

\end{remark}

\subsection{Simplical GKM graph complexes and equivariant cohomology}	
In this subsection, we introduce simplicial GKM graph complexes and their equivariant cohomology. Then we compute the equivariant cohomology ring of certain simplicial GKM orbifold complexes with integer coefficients.

We recall the definition of an axial function and connection on a regular graph $\Gamma$ from \cite[Definition 2.2]{DKS}. Let $\Gamma=(V, E)$ be a regular graph. We denote the edges in $\Gamma$ emanating from $p \in V(\Gamma)$ by $E_p(\Gamma)$. We also denote the starting vertex of $e \in E(\Gamma)$ by $s(e)$, and the terminal vertex by $t(e)$.   
 
 An axial function $\alpha$ on a regular graph $\Gamma$ is a map $\alpha \colon E(\Gamma) \rightarrow H^2(BT;\mathbb{Q})$ such that the vectors in $\{\alpha(e)~|~e \in E_p(\Gamma)\}$ are pairwise linearly independent for any $p\in V(\Gamma)$ as well as if  $e \in E(\Gamma)$ is an oriented edge and $\bar{e}$ its reverse orientation then 
 \begin{equation}\label{eq_axial_map}
 r_e\alpha(e)=\pm r_{\bar{e}}\alpha(\bar{e}) \in H^2(BT;\mathbb{Z}).
 \end{equation}
 for some positive integers $r_e$ and $r_{\bar{e}}$.

The collection of maps $\theta_e  \colon E_{s(e)}(\Gamma) \to E_{t(e)}(\Gamma)$ for $e \in E(\Gamma)$ is called a connection on a regular graph $\Gamma$ if $\theta_{e}(e)=\bar{e}$ and $\theta_{\bar{e}}=\theta_e^{-1}$ for all $e \in E(\Gamma)$. 

\begin{definition}
Let $\alpha$ and $\theta$ be an axial function and a connection on a regular graph $\Gamma$ respectively. Then the triple $(\Gamma,\alpha,\theta)$ is called a GKM graph if the following holds.
    \begin{enumerate}
        \item The vectors in the set $\alpha(E_p(\Gamma))$ are pairwise linearly independent for every ${p\in V(\Gamma)}$.
        
      \item  For any two edges $e ,e'\in E(\Gamma)$ such that $s(e)=s(e')$, there exists $c_{e,e'}\in \ZZ\setminus \{0\}$ satisfying $c_{e ,e'}(\alpha(\theta_e(e'))-\alpha(e')) \equiv 0 \text{ mod }r_e\alpha(e)$.
        
    \end{enumerate}
\end{definition}

\begin{definition}
 Let $\mathcal{G}=\{\Gamma_\beta :=(V_{\beta},E_{\beta})\}_{\beta\in \Lambda}$ be a simplicial graph complex and $(\Gamma_\beta, \alpha_\beta, \theta_\beta)$ GKM graphs for every $\beta\in \Lambda$. If $\alpha_{\beta}=\alpha_{\delta}$ on $E_{\beta}\cap E_{\delta}$ and  $\theta_{\beta}=\theta_{\delta}$ for each edge $e\in E_{\beta}\cap E_{\delta}$, then the collection $\{(\Gamma_{\beta},\alpha_{\beta},\theta_{\beta})\}_{\beta\in \Lambda}$ is called a simplicial GKM graph complex. We denote this collection by $(\mathcal{G}, \alpha, \theta)$ where $\alpha(e)=\alpha_{\beta}(e)$ if $e\in E_\beta$ and $\theta_e: E_p(\Gamma_\beta)\to E_{p'}(\Gamma_\beta)$ whenever $e$ is the edge joining $p$ and $p'$ in $\Gamma_\beta=(V_\beta,E_\beta)$.    
\end{definition}

One can obtain a simplicial GKM graph complex $(\mathcal{G}_\mathcal{K},\alpha,\theta)$ from a simplicial GKM orbifold complex $\mathcal{K}$ using Remark \ref{rem:spindle}. Let $\mathcal{K}$ be the collection of GKM orbifolds $\{X_\beta\}_{\beta \in \Lambda}$. So $\Gamma_\beta$ is obtained by considering the fixed points of the torus action on $X_\beta$ as the vertices and the spindles give the edges if they contain two fixed points. Note that one can find the associated axial function $\alpha$, the connection $\theta$, $c_{e,e'}$ and $r_e$ for each GKM orbifold $X$, see for example  \cite[Section 2]{DKS}. Therefore we get the collective data $\{(\Gamma_\beta, \alpha_\beta, \theta_\beta)\}_{\beta \in \Lambda}$ which gives $(\mathcal{G}_{\mathcal{K}},\alpha,\theta)$.

	

\begin{example}\label{ex:gkm_simp_graph}
Figure \ref{Fig_gkm_graph} is obtained by gluing an edge of the boundary of two triangles. It is not a regular graph. However, observe that it is a simplicial graph complex which is a collection of $3$ regular graphs. Let us denote this graph complex by $\mathcal{G}$. Then $V(\mathcal{G})=\{v_0,v_1,v_2,v_3\}$. Let $T=(S^1)^4$ and $y_1, \ldots, y_4$ be the standard basis of $H^2(BT;\ZZ)$. Define the axial function $\alpha \colon E(\mathcal{G}) \rightarrow H^2(BT;\mathbb{Q})$ by 
		$$ \alpha(v_1v_2)=((y_1+y_4)-\frac{c_2}{c_1}(y_1+y_3)), \quad \alpha(v_2v_1)=((y_1+y_3)-\frac{c_1}{c_2}(y_1+y_4)),$$  $$\alpha(v_2v_0)=((y_1+y_2)-\frac{c_0}{c_2}(y_1+y_4)), \quad
		\alpha(v_0v_2)=((y_1+y_4)-\frac{c_2}{c_0}(y_1+y_2)),$$ $$ \alpha(v_1v_0)=((y_1+y_2)-\frac{c_0}{c_1}(y_1+y_3)), \quad \alpha(v_0v_1)=((y_1+y_3)-\frac{c_1}{c_0}(y_1+y_2)).$$
		$$\alpha(v_3v_0)=((y_1+y_2)-\frac{c_0}{c_3}(y_2+y_3)), \quad
		\alpha(v_0v_3)=((y_2+y_3)-\frac{c_3}{c_0}(y_1+y_2)),$$
		$$ \alpha(v_1v_3)=((y_2+y_3)-\frac{c_3}{c_1}(y_1+y_3)), \quad \alpha(v_3v_1)=((y_1+y_3)-\frac{c_1}{c_3}(y_2+y_3)),$$ for some non-zero integers $c_0,\dots ,c_5$. 
		Let $r_{v_iv_j}=c_i$ and $r_{v_jv_i}=c_j$. Then $$r_{v_iv_j}\alpha(v_iv_j)=-r_{v_jv_i}\alpha(v_jv_i)\in H^2(BT;\ZZ).$$
		

Let $e=v_iv_j$ be an edge in $\mathcal{G}$. Define $\theta_e(v_iv_j)=v_jv_i$ and $\theta_e(v_iv_{\ell})=v_jv_{\ell}$ if $v_iv_{\ell}$ and $v_jv_{\ell}$ are edges in $\mathcal{G}$ for some $\ell\in\{0,1,2,3\}$ and $\ell\neq j$. Then $\theta_e $ is a connection on each regular graph in this simplicial graph complex.

 Let $e=v_iv_j$ and $e{'}=v_iv_{\ell}$ be the edges in $\mathcal{G}$ such that $s(e)=s(e{'})$. Define $c_{e,e{'}}:=c_j.c_{i} \in \mathbb{Z}-\{0\}$. Then it satisfies the condition $$c_{e,e{'}} (\alpha(\theta_e(e{'}))- \alpha(e{'}))\equiv 0~ \text{mod}~ r_e\alpha(e).$$ Thus $(\mathcal{G},\alpha,\theta)$ is a simplicial GKM graph complex.    
		
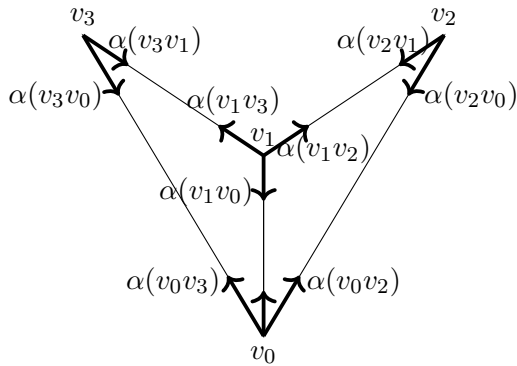
\begin{figure}
\begin{tikzpicture}[scale=.8]		
\draw (0,0)--(3,2)--(0,-3)--cycle;
				
\draw (0,0)--(-3,2)--(0,-3)--cycle;
				
\node[above] at (0, 0) {$v_1$};
				
\node[above] at (3,2) {$v_2$};
				
\node[below] at (0,-3) {$v_0$};
				
\node[above] at (-3,2) {$v_3$};
				
				\draw[->, line width=.5 mm] (0,0) to (.75,.5);
				
				\draw[->, line width=.5 mm] (0,0) to (0, -.75);
				
				\draw[->, line width=.5 mm] (0,0) to (-.75,.5);
				
				\draw[->, line width=.5 mm] (3,2) to (2.25, 1.5);
				
				\draw[->, line width=.5 mm] (-3,2) to (-2.25, 1.5);
				
				\draw[->, line width=.5 mm] (0, -3) to (0, -2.25);
				
				\draw[->, line width=.5 mm] (0, -3) to (.6, -2);
				
				\draw[->, line width=.5 mm] (0, -3) to (-.6, -2);
				
				\draw[->, line width=.5 mm] (3, 2) to (2.4, 1);
				
				\draw[->, line width=.5 mm] (-3, 2) to (-2.4, 1);

				\node [above] at (1, -.3) {$\alpha(v_1v_2)$};
				
				\node [above] at (-.5, .5) {$\alpha(v_1v_3)$};
				
				\node [above] at (2, 1.5) {$\alpha(v_2v_1)$};
				
				\node [above] at (-1.8, 1.5) {$\alpha(v_3v_1)$};
				
				\node [left] at (0, -.6) {$\alpha(v_1v_0)$};
				
				\node [above] at (1.5, -2.5) {$\alpha(v_0v_2)$};
				
				\node [above] at (-1.5, -2.5) {$\alpha(v_0v_3)$};
				
				\node [right] at (2.5, 1) {$\alpha(v_2v_0)$};
				
				\node [left] at (-2.5, 1) {$\alpha(v_3v_0)$};

			\end{tikzpicture}
			\caption{A simplicial GKM graph complex.}
			\label{Fig_gkm_graph}
		\end{figure}
		
	\end{example}

\begin{definition}
The equivariant cohomology ring of a simplicial GKM graph complex $(\mathcal{G},\alpha,\theta)$ is defined by $$H_T^{*}(\mathcal{G},\alpha,\theta):=\{f \colon V(\mathcal{G}) \rightarrow H^{*}(BT;\ZZ) ~~\big{|}~~ \widetilde{r}_e \alpha(e) ~\text{divides}~ (f(s(e))-f(t(e)))\},$$ where $\widetilde{r}_e$ is the smallest positive integer corresponding to  $e \in E(\mathcal{G})$ satisfying \eqref{eq_axial_map}. 
	\end{definition}

\begin{theorem}\label{Thm_DKS}
Let $\mathcal{K}$ be a compact simplicial GKM orbifold complex with respect to an effective action of the torus $T$ such that the geometric realization of $\mathcal{K}$ is homotopic to a $T$-CW complex. If all the isotropy subgroups of $T$ are connected and $H^{odd}(\mathcal{K};\mathbb{Z})=0$, then the equivariant cohomology ring $H_{T}^{*}(\mathcal{K};\mathbb{Z})$ is isomorphic to $H_{T}^{*}(\mathcal{G}_{\mathcal{K}}, \alpha, \theta)$ as $H_{T}^{*}(\{pt\};\mathbb{Z})$ algebra. 
\end{theorem}

\begin{proof}
The condition $H^{odd}(\mathcal{K};\ZZ)=0$ implies the Serre spectral sequence of the fibration $\mathcal{K} \xrightarrow{\iota} \mathcal{K}_{T} \xrightarrow{\pi} BT$ degenerate at $E_2$ page. Thus $H^*(\mathcal{K}_{T};\ZZ) \cong H^*(\mathcal{K};\ZZ) \otimes H^*(BT;\ZZ)$ as $\mathcal{K}$ is homotopic to finite $T$-CW complex and $T$ is compact. Since all the isotropy subgroups of this $T$-action are connected, we have the exactness of the Chang-Skjelbred sequence
$$ 0 \to H_{T}^{*}(\mathcal{K};\ZZ) \to H_{T}^{*}(\mathcal{K}_{0};\ZZ) \to H_{T}^{*+1}(\mathcal{K}_1,\mathcal{K}_{0};\ZZ),$$ where $\mathcal{K}_{0}$ denotes the set of all $T$-fixed point of $\mathcal{K}$ and $\mathcal{K}_{1}$ is the union of all zero and one-dimensional orbits in $\mathcal{K}$. Thus applying the arguments in the proof of \cite[Theorem 7.2]{GKM98} and \cite[Theorem 2.9]{DKS}, one can complete the proof for the simplicial GKM orbifold complex $\mathcal{K}$. 
	\end{proof}

\section{Equivariant Thom isomorphism}\label{sec_thom_iso}
At the beginning of this section, we discuss an explicit construction of an orbifold $G$-vector bundle. Then, we prove the $G$-equivariant Thom isomorphism for equivariant cohomology and equivariant $K$-theory for the orbifold $G$-vector bundles. This result generalizes \cite[Proposition 2.1]{SaSo}.

 Let $B$ be an effective orbifold and $\mathcal{A} := \{(\widetilde{V}_i,G_i,\phi_i)~|~i\in \mathcal{I}\}$ an orbifold atlas on $B$. Now assume that $(\widetilde{X}_i,\widetilde{V}_i,\widetilde{P}_i)$ is a $G_i$-invariant $\ell$-dimensional vector bundle for $i \in \mathcal{I}$ such that if there exists an embedding $\lambda \colon (\widetilde{V}_i,G_i,\phi_i) \rightarrow (\widetilde{V}_j,G_j,\phi_j)$ between orbifold charts then $\widetilde{X}_i =\lambda^{*}(\widetilde{X}_j)$, the pullback bundle along $\lambda$. Let  $\pi_i \colon \widetilde{X}_i \rightarrow X_i= \frac{\widetilde{X}_i}{G_i} $ be the orbit map. Then the triple $(\widetilde{X}_i, G_i, \pi_i)$ is an orbifold chart on $X_i$  for $i \in \mathcal{I}$.
	
If $ \phi_i(\widetilde{V}_i)\cap \phi_j(\widetilde{V}_j) \neq \emptyset$ for some $i, j \in \mathcal{I}$, there exists an orbifold chart $(\widetilde{V}_k,G_k,\phi_k) \in \mathcal{A}$ and two embeddings $$\lambda_1 \colon (\widetilde{V}_k,G_k,\phi_k) \to (\widetilde{V}_i,G_i,\phi_i) ~\mbox{and} ~\lambda_2 \colon (\widetilde{V}_k,G_k,\phi_k) \to (\widetilde{V}_j,G_j,\phi_j).$$ Then by the assumption $\widetilde{X}_k = \lambda_1^{*}(\widetilde{X}_i)  = \lambda_2^{*}(\widetilde{X}_j)$. Therefore, there exist two embeddings $$\widetilde{\lambda}_1 \colon (\widetilde{X}_k, G_k, \pi_k) \rightarrow (\widetilde{X}_i, G_i, \pi_i) ~\mbox{and}~ \widetilde{\lambda}_2 \colon (\widetilde{X}_k, G_k, \pi_k) \rightarrow (\widetilde{X}_j, G_j, \pi_j).$$

Note that the map $\widetilde{\lambda}_2 \widetilde{\lambda}_1^{-1} \colon \widetilde{\lambda}_1(\widetilde{X}_k) \rightarrow\widetilde{\lambda}_2(\widetilde{X}_k)$ is an equivariant diffeomorphism. Therefore, using the discussion in \cite[Section 1.3]{ALR}, 
we can construct an orbifold $X :=(\sqcup \frac{\widetilde{X}_i}{G_i}/ \sim)$ where $\pi_i(x_i)\sim \pi_j(x_j) $  if $x_j=\widetilde{\lambda}_2 \widetilde{\lambda}_1^{-1}(x_i)$ for  $x_i \in \widetilde{X}_i$ and $x_j \in \widetilde{X}_j$. The collection ${\{(\widetilde{X}_i,G_i,\pi_i) ~|~ i \in \mathcal{I}\}}$ is an orbifold atlas for $X$. Since $\widetilde{P}_i \colon \widetilde{X}_i  \rightarrow \widetilde{V}_i$ is $G_i$-invariant then it  induces a map $P_i \colon \frac{\widetilde{X_i}}{G_i} \rightarrow \frac{\widetilde{V}_i}{G_i} = \phi_i(V_i)$ for $i \in \mathcal{I}$. This gives an $\ell$-dimensional orbifold vector bundle $P \colon X \to B$, since $\{V_i=\phi_i(\widetilde{V}_i)~|~ i\in \mathcal{I}\}$ is an open covering of $B$. We denote this bundle by the triple $(X, B, P)$.
	
Let $G$ be a topological group. Then the vector bundle $\widetilde{P}_i \colon \widetilde{X}_i \rightarrow \widetilde{V}_i$ is called a $G$-vector bundle if $\widetilde{X}_i$ and $\widetilde{V}_i$ are $G$ spaces, $\widetilde{P}_i$ is a $G$-equivariant map and for each $g\in G$, the map $g \colon \widetilde{P}_i^{-1}(x)\to \widetilde{P}_i^{-1}(gx)$ is a linear map for $x\in \widetilde{V}_i$.

\begin{definition}\label{orbifold vector bundle}
Let $X$ and $B$ be two $G$-spaces such that  the orbifold vector bundle $\widetilde{P}_i \colon \widetilde{X}_i\rightarrow \widetilde{V}_i$ is a $G$-vector bundle and the action of $G$ (on $\widetilde{V}_i$ and $\widetilde{X}_i$) commutes with the action of $G_i$ for all $i\in \mathcal{I}$. Then the map $P \colon X \rightarrow B$ constructed above is called an orbifold $G$-vector bundle.   	
\end {definition}
		
\begin{remark}
The action of $G$ on $\widetilde{V}_i$ commutes with the action of $G_i$ for all $i\in \mathcal{I}$. Then this induces a $G$-action on $\frac{\widetilde{V}_i}{G_i}=V_i\subset B $ and the composition map $\phi_i \circ \widetilde{P}_i \colon \widetilde{X}_i \rightarrow V_i$ is also $G$-equivariant and this will induce the simple orbifold $G$-bundle $P_i \colon \frac{\widetilde{X}_i}{G_i} \rightarrow V_i$ which is considered in \cite[Section 2]{SaSo}. Here we study more general cases.
 \end{remark}
Let   $P\colon X \to B$ be an orbifold $G$-vector bundle as in Definition \ref{orbifold vector bundle}. Then, for each $i \in \mathcal{I}$, we have the disc bundle $D(\widetilde{X}_i) \to \widetilde{V}_i$ and the sphere bundle $S(\widetilde{X}_i) \to \widetilde{V}_i$. Since the group $G_i$ is finite, we get orbifold charts  $(D(\widetilde{X}_i), G_i,\pi_i)$ (as orbifold with boundary) and  $(S(\widetilde{X}_i), G_i,\pi_i)$ for $i \in \mathcal{I}$. Therefore, using the discussion in \cite[Section 1.3]{ALR}, we can construct an orbifold with boundary $D(X) \subset X$ and an orbifold $S(X) \subset D(X)$ such that the collection $\{(D(\widetilde{X}_i), G_i,\pi_i) ~|~ i\in \mathcal{I}\}$ and $\{(S(\widetilde{X}_i), G_i,\pi_i) ~|~ i \in \mathcal{I}\}$ are orbifold atlas for $D(X)$ and $S(X)$ respectively. 
The restrictions $P|_{D(X)} \colon D(X) \to B$ and $P|_{S(X)} \colon S(X) \to B$ are called respectively the $\mathbf{q}$-disc bundle and the $\mathbf{q}$-sphere bundle of $(X, B, P)$. Then, one can define the corresponding Thom space by $\text{Th}(X):=\frac{D(X)}{S(X)}$. The following result is proved in \cite[Section 2]{SaSo}. 

		
\begin{proposition}[Thom isomorphism for global quotient]\label{prop_thm_local}
Let $\mathcal{E}^*_G$ be one of $H^*_G$ and $K^*_G$ and $ P \colon (X/G_{f}) \to B$ be an $\ell$-dimensional simple orbifold $G$-bundle for some finite group $G_{f}$. Then $P$  induces the following isomorphism 
$$P^* \colon \mathcal{E}^*_G(B; \QQ) \to \mathcal{E}^{*+\ell}_G(\text{Th}(X/G_{f}); \QQ).$$
\end{proposition}

Next, we prove the Thom isomorphism for any orbifold $G$-vector bundles. 

\begin{proposition}[Thom isomorphism for orbifold $G$-vector bundle]\label{prop_thm_global}
Let $\mathcal{E}^*_G$ be one of $H^*_G$ and $K^*_G$, and  $P \colon X \to B$ an $\ell$-dimensional orbifold $G$-vector bundle as in Definition \ref{orbifold vector bundle}. Suppose that  $G$- and $G_i$-representations commute on each fiber of $\widetilde{P}_i \colon \widetilde{X}_i \to \widetilde{V}_i$ for each $i \in \mathcal{I}$. If $B$ is compact, then the map  $$P^* \colon \mathcal{E}^*_G(B;\QQ) \to \mathcal{E}^{*+\ell}_G(\text{Th}(X); \QQ)$$ is an isomorphism.
\end{proposition}

\begin{proof}
Since $B$ is compact, there is a finite open cover $V_1, \ldots, V_k$ such that each $V_i$ has an orbifold chart and each $V_i$ is $G$-invariant. So each restriction $P_i \colon \frac{{\widetilde{X}}_i}{G_i} \to V_i$ satisfies the hypothesis in Proposition \ref{prop_thm_local}. If $k=2$ (i.e. $B=V_1\cup V_2$) and $W:=V_1\cap V_2~ (\subset V_1)$ then uniformization of $V_1$ induces an orbifold chart for $W$. Thus $P_W \colon {X|_W} \to W$ also satisfies the hypothesis in Proposition \ref{prop_thm_local} where $X|_{W}= P^{-1}(W)$.  Now using the Mayer-Vietoris sequence, we get the following two horizontal long exact sequences such that each square of the following diagram commutes.
\[
\adjustbox{scale=.95,center}{
 \begin{tikzcd}
\mathcal{E}_G^{*-\ell-1}(W)\arrow{r} \arrow{d}{P^*_W} & \mathcal{E}_G^{*-\ell}(B) \arrow{r} \arrow{d}{P^*} &   \mathcal{E}_G^{*-\ell}(V_1)\oplus \mathcal{E}_G^{*-\ell}(V_2) \arrow{d}{P_1^* \oplus P_2^*} \arrow{r} & 	\mathcal{E}_G^{*-\ell}(W)\arrow{d}{P^*_W}\\
\mathcal{E}_G^{*-1}(\text{Th}(X|_{W}) \arrow{r} & \mathcal{E}_G^{*}(\text{Th}(X)) \arrow{r} & \mathcal{E}_G^{*}(\text{Th}(X|_{V_1})\oplus \mathcal{E}_G^{*}(\text{Th}(X|_{V_2}) \arrow{r} & \mathcal{E}_G^{*}(\text{Th}(X|_W).
\end{tikzcd}
}
\]

 Thus, using the five-lemma, we get that $ P^* \colon \mathcal{E}^*_G(B; \QQ) \to \mathcal{E}^{*+\ell}_G(\text{Th}(X); \QQ)$ is an isomorphism.  Now for $k>2$ one can complete the proof by the inductive arguments. 
\end{proof}

\section{Generalized equivariant cohomologies of simplicial GKM orbifold complexes}\label{sec_gen_cohom}
 In this section, we extend the result \cite[Theorem 2.3]{HHH} to a broader category of $G$-spaces equipped with singular invariant stratification.
 We give a presentation of the equivariant cohomology ring and the equivariant $K$-theory ring of a buildable simplicial GKM orbifold complex with rational coefficients. We describe the integral equivariant cohomology ring, equivariant $K$-theory ring and equivariant cobordism ring of a divisive simplicial GKM orbifold complex.
 
One can define the Thom class of 
an orbifold $G$-vector bundle  similarly to the usual $G$-vector bundle. Briefly, an element $u \in \mathcal{E}^*_G(\text{Th}(X))$ is called the Thom class of an orbifold $G$-vector bundle
$P \colon X \to B$ if for each closed subgroup $H$ of $G$ and $x \in B^{H}$, the restriction of $u$
to $X|_{G{\cdot}x}:=P^{-1}({G{\cdot}x})$ 
is a generator of $$\mathcal{E}^*_G(\text{Th}(X|_{G{\cdot}x}))
\cong \mathcal{E}^*_{H}(D(X|_{G{\cdot}x}), S(X|_{G{\cdot}x})).$$ Note that the Thom class $u$
is natural under pullback. The restriction of the Thom
class $u$ to the base $B$ via the zero section $s \colon B \to X$ is called the $G$-equivariant
Euler class $e_G(P) := s^*(u) \in \mathcal{E}^{\ast}_G(B)$.

Now we consider the following $G$-invariant stratification 
\begin{equation}\label{eq_g-tratification}
\{pt\}=Y_0 \subseteq Y_1 \subseteq Y_2 \subseteq \cdots 
\end{equation}
of a $G$-space $Y$ such that $Y= \bigcup_{j \geq 0} Y_j$ and each successive quotient
$Y_j/Y_{j-1}$ is homeomorphic to the Thom space $\text{Th}(X_j)$
of an orbifold $G$-vector bundle $\xi^j \colon X_j \to B_j$. 
Therefore, $Y$ can be built from $Y_0$ inductively by attaching $D(X_j)$ to $Y_{j-1}$ via some $G$-equivariant maps
$$\eta_j \colon S(X_j) \to Y_{j-1},$$ for $j\geq 1$.
This gives the following cofibration $$Y_{j-1} \to Y_j \to \text{Th}(X_j).$$ 
Then, one gets the following proposition about the generalized equivariant cohomologies with rational coefficients by induction on the stratification and by a similar argument to the proof of \cite[Theorem 2.3]{HHH}.

\begin{proposition}\label{prop_gen_cohom}
Let $Y$ be a $G$-space with a stratification as in \eqref{eq_g-tratification} and $\mathcal{E}^*_G=H^*_G ~\text{or}~ K^*_G$. Assume that each equivariant Euler class
$e_G(\xi^j) \in \mathcal{E}^*_G(B_j)$ for the orbifold $G$-vector bundle $\xi^j$
is not a zero divisor. Then the equivariant inclusion $\iota \colon \bigsqcup_{j\geq 0} B_j \hookrightarrow Y$ induces an injection 
\begin{equation}\label{eq_injectivity}
\iota^* \colon \mathcal{E}^*_G(Y) \to \bigoplus_{j\geq 0} \mathcal{E}^*_G(B_j).
\end{equation}
\end{proposition}

An approach to computing the image of $\iota^*$ in \eqref{eq_injectivity} has been discussed in \cite[Section 3]{HHH} when $\xi^j$'s are $G$-vector bundles. 

Let $Y$ be a $G$-space with the $G$-stratification as in \eqref{eq_g-tratification} which satisfies the following assumptions.

\begin{enumerate}
\item[(A1)] Each orbifold $G$-vector bundle $\xi^j \colon X_j \to B_j$ has the following decomposition 
$$(\xi^j \colon X_j \to B_j) \cong \bigoplus_{s < j} (\xi^{j s} 
\colon X_{js} \to B_j)$$ into orbifold $G$-vector bundles $\xi^{j s}$, (where $X_{j s}$ can be trivial).

\item[(A2)] The restriction of the attaching map $\eta_j \colon S(X_j) \to Y_{j-1}$ on $S(X_{js})$ satisfies
 $${\eta_j}|_{S(X_{j s})} = f_{j s} \circ \xi^{j s}$$ 
for some $G$-equivariant map $f_{j s} \colon B_j \to B_{s}\subset Y_{j-1} $, for $s<j$. 

\item[(A3)] The equivariant Euler classes $\{e_G(\xi^{js}) ; s<j\}$ are not zero divisors
and pairwise relatively prime in $\mathcal{E}^{\ast}_G(B_{j})$. 
\end{enumerate} 

The map $f_{js}$ induces $f_{js}^* \colon \mathcal{E}^*_{T}(B_s; \QQ) \to \mathcal{E}^*_{T}(B_j; \QQ)$. 
Note that under the above assumptions
on a $G$-space $Y$ with the property as in \eqref{eq_g-tratification}, 
one may obtain the following result with rational coefficients. 

\begin{proposition}\label{prop:equivariant_k-theory}
Let $Y$ be a $G$-space with a $G$-stratification as in \eqref{eq_g-tratification} such that assumptions {\rm (A1)},{\rm (A2)} and {\rm (A3)} are satisfied.
Then the image of $\iota^* \colon \mathcal{E}^*_G(Y) \to \bigoplus_{j \geq 0} \mathcal{E}^*_G(B_j)$ is
$$\Gamma_Y:= \Big \{ (x_j) \in \bigoplus_{j \geq 0} \mathcal{E}^*_{G}(B_j) ~ \big{|} ~ e_G(\xi^{j s})~
\mbox{divides} ~ x_j - f^*_{j s}(x_{s}) ~\mbox{for all}~ s < j\Big \}.$$
\end{proposition} 
 
 \begin{proof}
Using Proposition \ref{prop_thm_global}, one can get the proof by the arguments similar to the proof of \cite[Theorem 3.1]{HHH} and \cite[Proposition 2.3]{SaSo}. 
 \end{proof}

\begin{remark}\label{rmk_non_abe_case}
If $G$ is non-abelian, one can get a similar conclusion as in Proposition \ref{prop:equivariant_k-theory} whenever the $G$-space has a $G$-stratification as in \eqref{eq_g-tratification} and satisfies the conditions similar to $(A1), (A2)$ and $(A3)$. Also, if $\xi^j$'s are $G$-vector bundles, then the hypothesis in Proposition \ref{prop:equivariant_k-theory} is the same as the hypothesis \cite[Theorem 3.1]{HHH}.   
\end{remark}
 
Now we study some generalized equivariant cohomologies of simplicial GKM orbifold complexes.

 \begin{lemma}\label{prop_gkm_cond}
 The $T$-invariant stratification of a buildable simplicial GKM orbifold complex satisfies the conditions (A1), (A2) and (A3).
 \end{lemma}
 \begin{proof}
  Let $\mathcal{K}$ be a  buildable simplicial GKM orbifold complex. Then, by Definition \ref{def_filt_simp_gkm}, we have a filtration $$\{pt\}=\mathcal{K}_0\subset \mathcal{K}_1\subset \mathcal{K}_2\subset \cdots \subset \mathcal{K}_m=\mathcal{K}$$
such that $\mathcal{K}_j$ is closed, $h^{-1}(E_j\setminus E_{j-1})$ is the one skeleton of $ \mathcal{K}_j\setminus \mathcal{K}_{j-1}$ and  $M_j := \mathcal{K}_j\setminus \mathcal{K}_{j-1}$ is $T$-equivariantly homeomorphic to $\CC^{d_j}/L_j$ for some 
finite group $L_j$ for $j=0,1,2,\dots,m$.
Here, $d_j$ is same as the number of edges adjacent to $b_j$ in $F_j=E_j \setminus E_{j-1}$.
Consider the $G$-vector bundle $\psi^j \colon \CC^{d_j}/L_j \to h^{-1}(b_j) $ over $h^{-1}(b_j)$. Note that $d_j\leq j$ follows from the filtration on the simplicial graph complex  $\mathcal{G}$ in Definition \ref{def_filt_simp_gkm}. Let $\{\alpha_{j_s}\}_{s=1}^{d_j}$ be the characters corresponding to the torus action on $h^{-1}(F_j)$ and $\mathcal{V}_j:=\bigoplus_{s=1}^{d_j} \mathcal{V}(\alpha_{j_s})$. Then $\mathcal{V}_j$ is the tangent space of $M_j$ at $h^{-1}(b_j)$ and $\psi^j \colon \CC^{d_j}/L_j \to h^{-1}(b_j)$ is equivariantly homeomorphic to $\xi^j \colon \mathcal{V}_j\to h^{-1}(b_j) $. Now the following gives the assumption (A1).
\begin{align*}
  \xi^j \colon \mathcal{V}_j\to h^{-1}(b_j)
& \cong \xi^j \colon \bigoplus_{s =1}^{d_j} \mathcal{V}(\alpha_{j_s}) \to h^{-1}(b_j)\\
&=\bigoplus_{s=1}^{d_j}\big{(}\xi^{js} \colon \mathcal{V}(\alpha_{j_s}) \to h^{-1}(b_j)\big{)}.
\end{align*}

 Considering
 $f_{js} \colon h^{-1}(b_j) \to h^{-1}(b_s)$ as a map between two fixed points, we conclude the assumption (A2).

 The Euler class of $\xi^{js} \colon \mathcal{V}(\alpha_{j_s}) \to h^{-1}(b_j)$ is determined by $e_{G}(\xi^{js})=\alpha_{j_s}$.
Since $\{\alpha_{j_s}\}_{s=1}^{d_j}$ are pairwise linearly independent, hence the assumption (A3) follows. 
  \end{proof}

\begin{theorem}\label{thm:buildable_gkm}
Let $\mathcal{K}$ be a buildable simplicial GKM orbifold complex and $\mathcal{E}^*_T\in\{ H^*_T,K^*_T\}$. Then the generalized $T$-equivariant cohomology of $\mathcal{K}$  can be given by
 $$\mathcal{E}^*_T(\mathcal{K} ; \QQ) = \Big \{ (x_j) \in \bigoplus_{j=0}^{m} \mathcal{E}^*_{T}(\{pt\}; \QQ) ~ \big{|} ~ e_T(\xi^{js})~
	\mbox{divides} ~ x_j - f^*_{j s}(x_{s}) ~\mbox{for all}~ s < j\Big \}.$$
\end{theorem}

\begin{proof}
By Lemma \ref{prop_gkm_cond}, the $T$-invariant stratification of $\mathcal{K}$ as in \eqref{eq_G-tratification} satisfies conditions (A1), (A2) and (A3). Here $B_j=h^{-1}(b_j)$  corresponds to a fixed point of the $T$-action on $\mathcal{K}$ for $ j \in \{0, 1, \ldots, m\}\}$. Then, the proof follows from Proposition \ref{prop:equivariant_k-theory}.
\end{proof}
 
Next, we discuss the computation of $\mathcal{E}_T^*(\mathcal{K};\ZZ)$ of a divisive simplicial GKM orbifold complex $\mathcal{K}$ with integral coefficients when $\mathcal{E}^*_T= H^*_T$, $K^*_T$ or $MU^*_T$. Note that if  $\mathcal{K}$ is a divisive simplicial GKM orbifold complex then similarly to the proof of Lemma \ref{prop_gkm_cond}, one can show that it has a $T$-invariant stratification $$\{pt\}= \mathcal{K}_0 \subset \mathcal{K}_1 \subset \mathcal{K}_2 \subset \cdots \subset \mathcal{K}_{m} =\mathcal{K},$$ where each $\mathcal{K}_j\setminus \mathcal{K}_{j-1}$ is an invariant cell which is equivariantly homeomorphic to $\CC^{d_j}$ for $j=1,2,\dots,m$. So a divisive simplicial GKM orbifold complex is integrally equivariantly formal.  Moreover, this stratification satisfies conditions (A1), (A2) and (A3). Therefore, using Remark \ref{rmk_non_abe_case}, we can get the following result with integer coefficients.

\begin{theorem}\label{thm:buildable_gkm_2}
Let $\mathcal{K}$ be a divisive simplicial GKM orbifold complex and $\mathcal{E}^*_T\in\{H^*_T,K^*_T,MU^*_T\}$. Then the generalized $T$-equivariant cohomology of $\mathcal{K}$ with integer coefficients can be given by
 $$\mathcal{E}^*_T(\mathcal{K} ; \ZZ) = \Big \{ (x_j) \in \bigoplus_{j=0}^{m} \mathcal{E}^*_{T}(\{pt\}; \ZZ) ~ \big{|} ~ e_T(\xi^{js})~\mbox{divides} ~ x_j - f^*_{j s}(x_{s}) ~\mbox{for all}~ s < j\Big \}.$$
\end{theorem}

\begin{remark}
 Gonzales studies `$\QQ$-filterable spaces' in \cite{Gon}. If a $\QQ$-filterable space is a projective $T$-variety then it has a stratification similar to \eqref{eq_G-tratification} where $\mathcal{K}_j \setminus \mathcal{K}_{j-1}$ is a `rational cell'. Under the assumption of `$T$-skeletal', he studies the GKM theory of $\QQ$-filterable spaces to obtain their Borel equivariant cohomology rings. We also compute some other generalized equivariant cohomology rings, like equivariant $K$-theory and equivariant cobordism rings of these spaces. 
\end{remark}	

Now we give an example of a simplicial GKM orbifold complex such that the simplicial GKM graph complex associated with it is the same as the simplicial GKM graph complex of Example \ref{ex:gkm_simp_graph}. 
\begin{example}\label{ex:equi_cohom_gkmplx}
	Let $\mathcal{K}$ be the collection of 3 GKM orbifolds $\mathcal{O}_1,\mathcal{O}_2,\mathcal{O}_3$ where $$\mathcal{O}_1:=\{[z_0:\cdots:z_5]\in \WW P(c_0,\ldots,c_5)~|~ z_3=z_4=z_5=0\} \cong \WW P(c_0,c_1,c_2),$$
	$$\mathcal{O}_2:=\{[z_0:\cdots:z_5]\in \WW P(c_0,\ldots,c_5)~|~ z_2=z_4=z_5=0\} \cong \WW P(c_0,c_1,c_3)\mbox{ and}$$
	$$\mathcal{O}_3:=\{[z_0:\cdots:z_5]\in \WW P(c_0,\ldots,c_5)~|~ z_2=z_3=z_4=z_5=0\} \cong \WW P(c_0,c_1).$$ 
 
 Then $\mathcal{O}_3=\mathcal{O}_1\cap \mathcal{O}_2$. Consider the action of $T^4$ on $\WW P(c_0,\ldots,c_5)$ is defined by
 \begin{equation}\label{weight_T}
 (t_1,t_2,t_3,t_4)[z_0:\cdots :z_5]=[t_1t_2z_0:t_1t_3z_1:t_1t_4z_2:t_2t_3z_3:t_2t_4z_4:t_3t_4z_5 ].
 \end{equation}
 This induces $T^4$-actions on $\mathcal{O}_1,\mathcal{O}_2 \mbox{ and } \mathcal{O}_3$. The spaces $\mathcal{O}_1,\mathcal{O}_2 \mbox{ and } \mathcal{O}_3$ are all GKM orbifolds with the above $T^4$-action and hence $\mathcal{K}$ is a simplicial GKM orbifold complex.

Let $[e_i] \in \WW P(c_0,\ldots,c_5)$ such that the only non-zero entry is the $i$th coordinate. So, $[e_i]$ is a fixed point of the $T^4$-action for $i=1, \ldots, 5$. Let $v_i$ be the vertex corresponding to $[e_i]$.
  The graph corresponding to $\mathcal{O}_1$ is the triangle $v_0v_1v_2$ and the graph corresponding to $\mathcal{O}_2$ is the triangle $v_0v_1v_3$ and they have a common edge $(v_0v_1)$ which is the graph of $\mathcal{O}_3$. 

Thus the simplicial graph complex associated to $\mathcal{K}$ is the same as the simplicial GKM graph complex $\mathcal{G}$ in Example \ref{ex:gkm_simp_graph}. Now if $c_i$ divides $c_{i-1}$ for all $i\in\{1,2,3\}$ then by \cite[Theorem 3.19]{BS_Gra}, one can show that $\mathcal{K}$ has a $T^4$-invariant cell-structure and $H^{odd}(\mathcal{K};\ZZ)=0$. Therefore, by Theorem \ref{thm:buildable_gkm_2} 
	\begin{align}\label{app_dks}
		 \mathcal{E}_{T^4}^{*}(\mathcal{K};\mathbb{Z})  = &\{(x_j)\in \bigoplus_{j=0}^{3}\mathbb{Z}[y_1,y_2,y_3,y_4] ~\big{|}~e_{T^4}(\xi^{js})~\mbox{divides}~(x_j-x_s)\\ & \text{ if }s<j <3 ~\mbox{and} ~(s,j) \neq (2, 3)\} \nonumber.  
	\end{align}
 Next, we discuss how to compute $e_{T^4}(\xi^{js})$. Using \eqref{weight_T} the character of the one-dimensional representation for the bundle $\xi^{js}$ is given by
 \begin{equation}\label{eq_one_dim_rep}
 	(t_1,t_2,t_3,t_4)\to \mathcal{T}_{s}(\mathcal{T}_j)^{-\frac{ c_s}{ c_j}},
 \end{equation}
for $s<j<3$ and $(s,j)\neq (2,3)$, where $\mathcal{T}_0=t_1t_2, \mathcal{T}_1=t_1t_3, \mathcal{T}_2=t_1t_4, \mathcal{T}_3=t_2t_3$. 
 
 We recall that $H_{T^4}^*(\{pt\};\ZZ)=H^{*}(BT^4;\ZZ)\cong \ZZ[y_1, y_2,y_3, y_4]$ where $y_1, y_2,y_3,y_4$ be the standard basis of $H^2(BT^4;\ZZ)$. Also
 $K^*_{T^4} (\{pt\};\ZZ) \cong R(T^4)[z, z^{-1}],$ where $R(T^4)$ is the complex representation ring of $T^4$ and $z$ is the Bott element in $K^{-2}(\{pt\})$. Note that $R(T^4)$ is isomorphic to the ring of Laurent polynomials with $4$-variables, that is $R(T^4) \cong \ZZ[\zeta_1, \ldots, \zeta_4]_{(\zeta_1 \cdots \zeta_4)}$, where  $\zeta_i$ is the irreducible representation corresponding to the projection on the $i$-th factor, see \cite{Hus}. Let  $\mathcal{Y}_0:=y_1+y_2, \mathcal{Y}_1:=y_1+y_3, \mathcal{Y}_2:=y_1+y_4, \mathcal{Y}_3:=y_2+y_3$ and
 $\mathcal{Z}_0:=\zeta_1\zeta_2, \mathcal{Z}_1:=\zeta_1\zeta_3, \mathcal{Z}_2:=\zeta_1\zeta_4, \mathcal{Z}_3:=\zeta_2 \zeta_3$. Therefore, using \eqref{eq_one_dim_rep}, one has the following.
 \begin{align}\label{eq_eu_cl}
e_{T^4}(\xi^{js}) = 
     \begin{cases}
     \mathcal{Y}_{s} - d_{js}\mathcal{Y}_{j} & \quad\text{in}~ H^2_{T^4}(\{pt\}; \ZZ)\\
       1-  \mathcal{Z}_{s} \mathcal{Z}_{j}^{-d_{js}} & \quad\text{in}~  K^0_{T^4}(\{pt\}; \ZZ) \\
    e_{T^4}(\mathcal{Z}_{s} \mathcal{Z}_{j}^{d_{js}})  &\quad\text{in}~ MU^2_{T^4}(\{pt\}; \ZZ),
     \end{cases}
\end{align} 
where $d_{js}=\frac{c_s}{c_j} \in \ZZ$ for $s<j <3$ and $(s,j) \neq (2, 3)$.
\end{example}

\begin{lemma}\label{lem_module1}
Let $\mathcal{K}$ be a divisive simplicial GKM orbifold complex. Then, for $\mathcal{E}^*_{T} = H^*_{T}, K^*_{T}, ~\mbox{or} ~ MU^*_{T}$, there exists a non-canonical isomorphism of $\mathcal{E}_T^{*}(\{pt\};\ZZ)$-module
     $$\mathcal{E}_T^{*}(\mathcal{K};\ZZ)\cong \prod_{j=0}^m e_T(\xi^j)\mathcal{E}_T^{*}(\{pt\};\ZZ).$$
\end{lemma}
\begin{proof}
 This follows from Lemma \ref{prop_gkm_cond}, Theorem \ref{thm:buildable_gkm_2} and the second part of \cite[Theorem 2.3]{HHH}.   
\end{proof}

\begin{theorem}\label{thm_module2}
    Let $\mathcal{K}$ be a divisive simplicial GKM orbifold complex and let $b_0<b_1<b_2<\dots<b_m$ be the ordering of the vertices in simplicial GKM graph complex. For each $j\in\{0,1,\dots,m\}$ let $\phi_j\in \mathcal{E}_T^{*}(\mathcal{K};\ZZ)$ be an element such that $(\phi_j)_s=0 \text{ for } s<j$ and $(\phi_j)_j$ is a scalar multiple of $e_T(\xi^j)$. Then the set $\{\phi_j\}_{j=0}^m$ generates $\mathcal{E}_T^{*}(\mathcal{K};\ZZ)$ freely as a module over $\mathcal{E}_T^{*}(\{pt\};\ZZ)$.
\end{theorem}
\begin{proof}
 This follows from Lemma \ref{prop_gkm_cond}, Theorem \ref{thm:buildable_gkm_2} and  \cite[Theorem 4.1]{HHH}. 
\end{proof}

\begin{remark}
Lemma \ref{lem_module1} and Theorem \ref{thm_module2} hold for a buildable simplicial GKM orbifold complex for $\mathcal{E}^*_{T} \in \{H^*_{T}, K^*_{T}\}$  with rational coefficients.    
\end{remark}







\vspace{1cm} 

 \noindent  {\bf Acknowledgment.} 
 The first author thanks `Indian Institute of Technology Madras' for PhD fellowship. The second author thanks `Science and Engineering Research Board India' for research grants.

\bibliographystyle{abbrv}

\end{document}